\documentclass[final]{siamart171218}

\usepackage{amsfonts,amsmath}
\usepackage{amssymb}
\newtheorem{remark}[theorem]{Remark}

\newcommand{\E}{\mathbb{E}}
\newcommand{\R}{\mathbb{R}}
\newcommand{\N}{\mathbb{N}}
\newcommand{\norm}[1]{\left\|#1\right\|}

\newcommand{\Prob}{{\mathbb P}}
\newcommand{\by}{{\bf y}}
\newcommand{\byd}{{\bf y}^\delta}
\newcommand{\be}{{\boldsymbol \epsilon}}
\newcommand{\fmnd}{f_{m,N}}
\newcommand{\est}{g_{m,N}^{R}}

\newcommand{\cone}{D_1}
\newcommand{\ctwo}{D_2}
\newcommand{\cthree}{D_3}

\newcommand{\dmeas}{\nu}
\newcommand{\jmeas}{\mu}
\newcommand{\likemeas}{\rho}
\newcommand{\bdmeas}{{\boldsymbol \dmeas}}

\DeclareMathOperator*{\argmin}{\rm arg\, min}
\newcommand{\dom}{{\mathcal D}}
\def\tr{{\rm Tr}}
\def\HS{{\rm HS}}

\title{Least squares approximations in linear statistical inverse learning problems}

\author{
T. Helin\footnotemark[1]
}

\begin{document}

\maketitle

\renewcommand{\thefootnote}{\fnsymbol{footnote}}
\footnotetext[1]{LUT University, School of Engineering Science, P.O.~Box 20, FI-53851 Lappeenranta, Finland. The work of TH was supported by the the Academy of Finland (decision 326961). The author is indebted to Vesa Kaarnioja who provided valuable input to the project.}

\begin{abstract}
Statistical inverse learning aims at recovering an unknown function $f$ from randomly scattered and possibly noisy point evaluations of another function $g$, connected to $f$ via an ill-posed mathematical model. 
In this paper we blend statistical inverse learning theory with the classical regularization strategy of applying finite-dimensional projections.
Our key finding is that coupling the number of random point evaluations with the choice of projection dimension, one can derive probabilistic convergence rates for the reconstruction error of the maximum likelihood (ML) estimator. Convergence rates in expectation are derived with a ML estimator complemented with a norm-based cut-off operation. Moreover, we prove that the obtained rates are minimax optimal.
\end{abstract}

\begin{keywords}
Statistical learning, inverse problems, least squares approximations, minimax optimality.
\end{keywords}

\begin{AMS}
62G08, 62G20, 65J20, 68Q32.
\end{AMS}

\section{Introduction}

Statistical inverse learning aims at recovering an unknown function $f$ from randomly scattered and possibly noisy point evaluations of another function $g$, connected to $f$ via an ill-posed mathematical model. 
Statistical learning has a long tradition in inverse problems going back to the works \cite{o1990convergence, bissantz2004consistency} and has recently gained increasing attention in literature. 
A crucial component in addressing statistical inverse learning is the choice of regularization scheme \cite{engl1996regularization} needed to stabilize the inverse problem.

The success of a given inverse learning method is often described in (probabilistic) terms of the reconstruction error and, in particular, its convergence speed with respect to increasing number of point evaluations. To this effect, we highlight the recent work by Blanchard and M\"ucke \cite{blanchard2018optimal}, where minimax optimal convergence rates are derived for the general spectral regularization approach in Hilbert spaces under certain classes of sampling measure. The spectral approach has since been extended to non-linear inverse problems \cite{rastogi2019convergence}, adaptive parameter choice rules \cite{lu2020balancing} and convex regularization penalties \cite{bubba2021convex}.

In this paper we blend statistical inverse learning theory with the approach of regularization by projection. This approach is based on the idea that projecting either the domain or the range of the forward operator to a finite-dimensional subspace stabilizes the inverse problem \cite{engl1996regularization}. As one key motivation for studying projection-based strategies, common iterative algorithms such as the Krylov space methods can be interpreted as projection-based regularization methods for inverse problems. 

It is well-known that convergence rates for general inverse problems cannot be shown without further assumptions or applying projection in \emph{both} the domain and the range of the forward operator \cite{seidman1980nonconvergence}. 
Our key finding is that the statistical learning framework with finite number of random data point evaluations coupled with general finite-dimensional projection in operator domain provides similar remedy to the convergence study. When coupling the choice of projection dimension
with the number of random point evaluations, we are able to derive convergence rates for the probabilistic reconstruction error of the maximum likelihood (ML) estimator in the range of the projection. Moreover, convergence rates for the expected reconstruction error are derived for the ML estimator complemented with a norm-based cut-off. To complete the picture, we prove that the attained rates are minimax optimal.

In terms of frequentist approach to statistical inverse problems, our result is related to work by Mathe and Pereverzev \cite{mathe2001optimal}, who provide convergence rates for optimal discretization schemes in linear models. In the spirit of statistical learning, the projection applied in the range in \cite{mathe2001optimal} could be interpreted as taking dual pairings in a reproducing kernel Hilbert space with the corresponding kernel function thus giving rise to point evaluation data. However, the design in \cite{mathe2001optimal} is fixed while in learning context (including this paper) it is considered to be randomly generated by an unknown distribution. %This additional uncertainty requires substantially different mathematical tools for addressing the probabilistic concentration of the solution as we will illustrate below.

Let us describe this phenomena more rigorously. Suppose $\dom \subset \R^d$ is a Borel set and $\dmeas$ is a probability measure on $\dom$, which we will occasionally refer to as the \emph{design measure}. Let ${\mathcal H}$ be a separable real infinite-dimensional Hilbert space. We investigate the measurement model
\begin{equation}
\label{eq:main_model}
h=Af,
\end{equation}
where $h\in L^2(\dom,\dmeas)$ is the datum, $A\!:{\mathcal H}\to H_k \subset L^2(\dom,\dmeas)$ is a compact, \emph{one-to-one} linear operator, and $f\in {\mathcal H}$ is the unknown function.
The range $H_k$ is assumed to be a reproducing kernel Hilbert space (RKHS) induced by the positive semidefinite kernel $k : \dom\times \dom \to \R$. Moreover, without loss of generality we assume $\norm{A}_{{\mathcal L}(H, H_k)}\leq 1$.

Let $(x_n)_{n=1}^N\subset \dom$ be i.i.d.~with respect to the probability measure $\dmeas$. We are interested in finding an approximation for the ground truth $f^\dagger$ based on an ensemble of noisy observations $\byd = (y_n^\delta)_{n=1}^N \in \R^N$ such that
\begin{equation}
\label{eq:meas_model}
y_n^\delta=h^\dagger(x_n)+ \delta \epsilon_n,\quad n =1,\ldots,N,
\end{equation}
where $h^\dagger = Af^\dagger$ stands for the noise-free data corresponding to the ground-truth value $f^\dagger \in {\mathcal H}$, $\delta>0$ describes noise level and $\epsilon_n \sim {\mathcal N}(0,1)$ are independent and normally distributed. By denoting $X = \{x_n\}_{n=1}^N \subset \dom$, we reformulate \eqref{eq:meas_model} into a vectorized form 
\begin{equation}
	\label{eq:meas_model_vec}
	\by = S_X A f^\dagger + \delta \be \in \R^N,
\end{equation}
where $S_X : H_k \to \R^N$ is the evaluation operator at point set $X$ and $\be = (\epsilon_n)_{n=1}^N \in \R^N$. 
%is zero-mean random variable such that for some constants $\sigma, \eta>0$ we have
%\begin{equation}
%	\label{eq:noise_assumption}
%	\E |\epsilon_n|^p \leq \frac 12 p! \sigma^2 \tau^{p-2}
%\end{equation}
%for any $p\geq 2$. We note that a Gaussian distribution $\epsilon_n \sim {\mathcal N}(0,\delta^2)$ satisfies the bound \eqref{eq:noise_assumption} with $\sigma = \tau = \delta$.

%Define now a weighted function space $Y_\dmeas = L^2(\dom, \dmeas)$ as a Hilbert space induced by the inner product
%\begin{equation*}
%	\langle g_1, g_2 \rangle_{Y_\dmeas} := \int_\dom g_1(x) g_2(x) \dmeas(dx).
%\end{equation*}
To study the limit of increasing $N$, we denote
\begin{equation*}
	A_\dmeas = \iota A : {\mathcal H} \to L^2(\dom, \dmeas),
\end{equation*}
where $\iota : H_k \to L^2(\dom, \dmeas)$ is the canonical injection map, and introduce the corresponding normal operator
\begin{equation}
	\label{eq:Brho}
	B_\dmeas = A_\dmeas^* A_\dmeas : {\mathcal H} \to {\mathcal H}.
\end{equation}
For more properties of the operator $B_\dmeas$, see \cite[Prop. 19]{de2006discretization}.

Below, a key structure is the underlying discretization scheme of ${\mathcal H}$, which will stay fixed throughout the paper.
\begin{definition}
\label{def:admissible}
Let $V_m$, $m\geq 1$, be finite-dimensional subspaces of ${\mathcal H}$ such that $\dim V_m=m$ and let $P_m : {\mathcal H} \to V_m \subset {\mathcal H}$ be an orthogonal projection.
We call a sequence $\{V_m\}_{m=1}^\infty$ \emph{admissible} subspaces if $V_m \subset V_{m+1}$ for all $m\in \N$ and $\overline{\cup_{m=1}^\infty V_m} = {\mathcal H}$.
\end{definition}
The main purpose of definition \ref{def:admissible} is to make sure that any $f \in {\mathcal H}$ can be approximated for a given accuracy in some $V_m$ for $m$ large enough. 
We note that the nested structure of the subspaces is only utilized in the proof of the minimax optimality result below.

We seek an approximate solution for the ground truth $f^\dagger \in {\mathcal H}$ by defining the ML estimator as the minimum-norm least squares solution to $S_X A f = \byd$. More precisely, we set
\begin{equation}
	\label{eq:def_fmnd}
	\fmnd = \argmin_{f\in V_m} \left\{\norm{f}_{\mathcal H} \; | \; f \; \text{minimizes}\; \norm{S_X A f - \byd}_N \right\} \quad \text{a.s.},
\end{equation}
where $\|\cdot\|_N$ stands for the norm induced by the empirical inner product $\langle {\bf x}, {\bf z}\rangle_N = \frac 1N \sum_{n=1}^N x_n z_n$ with ${\bf x}, {\bf z} \in \R^N$. 
The estimator $\fmnd$ is defined up to a zero-measurable set as it is always unique and can be represented by a linear mapping applied to $\byd$ as we will see later.
%It follows from our assumptions on $\nu$, injectivity of $A$ and strict positive-definiteness of $k$ that $S_X A|_{V_m}$ is injective almost surely, and, therefore, $\fmnd$ can be defined up to a zero-measurable set.

In statistical inverse problems, it is well-known that further restrictions regarding the ground truth $f^\dagger$, so-called source conditions, are needed in order to derive concentration rates for regularized estimators \cite{cavalier2008nonparametric}.
Classical source conditions often imply certain smoothness of $f^\dagger$ via the mapping properties of $A$. Here, we can impose more explicit approximation conditions to $f^\dagger$ by connecting the source set $\Theta$ to the approximation rates obtained in subspaces $V_m$ by defining
\begin{equation}
	\label{eq:source}
	\Theta(s,R) = \{f \in {\mathcal H} \; | \;  \norm{(I-P_m) f}_{{\mathcal H}} \leq R(m+1)^{-s}\; \text{for all}\;  m\geq 0\} \subset {\mathcal H},
\end{equation}
where $s, R > 0$ and $V_m \subset {\mathcal H}$, $m\geq 1$, are admissible subspaces. Above, we use convention $P_0 = 0$. For example, finite-element approximations in Sobolev space ${\mathcal H} = H^1$ are typically bounded by some higher order Sobolev norm $H^p$, $p>1$, of the function and a power of the mesh size \cite{reddy2019introduction}. Here, this would correspond to $\Theta(s,R)$
coinciding with an $H^p$-Sobolev ball with parameter $s$ dependent on $p$ and the dimension of the domain.

To specify assumptions on the design measure let ${\mathcal P}(\dom)$ denote all probability measures on domain $\dom \subset \R^d$. We introduce the following parametrized subset of design measures
\begin{eqnarray}
	\label{eq:Pids}
	& {\mathcal P}^{>}(t, C) & =  \left\{\dmeas \in {\mathcal P}(\dom) \; \big| \; \lambda_{min}(P_m B_\dmeas P_m) \geq C m^{-t} \; \text{for all}\;  m\in\N\right\} \nonumber\\
	& {\mathcal P}^{<}(t, C) & =  \left\{\dmeas \in {\mathcal P}(\dom) \; \big| \; \lambda_{min}(P_m B_\dmeas P_m) \leq C m^{-t} \; \text{for all}\;  m\in\N\right\} 	
	\quad \text{and} \\
		&  {\mathcal P}^\times(C) & = \left\{\dmeas \in {\mathcal P}(\dom) \; \big| \; \norm{(P_m B_\dmeas P_m)^+ B_\dmeas (I-P_m)} \leq C\; \text{for all}\;  m\in\N\right\},\nonumber
\end{eqnarray}
where $\lambda_{min}(P_m B_\dmeas P_m)$ stands for the smallest \emph{non-zero} eigenvalue, i.e., the smallest eigenvalue of the finite-dimensional restriction $P_m B_\dmeas|_{V_m} : V_m \to V_m$. Also, recall that $\lambda_{min}(P_m B_\dmeas P_m)$ is always positive as $P_m B_\dmeas P_m$ is strictly positive-definite on $V_m$ for any $m\in\N$.

Restricting $\nu$ to the subsets defined in \eqref{eq:Pids} quantifies the ill-posedness of the statistical inverse problems. Good intuition is perhaps obtained by recalling that $B_\dmeas$ corresponds to the limit of infinite observations. For example, the set ${\mathcal P}^{>}$ specifies the worst instability of this limiting problem all across the subspaces $V_m$. We assume first $\dmeas \in {\mathcal P}^{>} \cap {\mathcal P}^\times$ to derive an \emph{upper bound} with given parameter choice rules. Then, a \emph{lower bound} to concentration is given in  ${\mathcal P}^{<} \cap {\mathcal P}^\times$, i.e., in a set restricting the stability. We only consider polynomial decay rates, i.e., mildly ill-posed problems, for convenience.

Next the probabilistic reconstruction error is characterized by the following theorem.

\begin{theorem}
\label{thm:main_result_prob}
Let $\{V_m\}_{m=1}^\infty$ be a sequence of admissible subspaces, $B_\nu$ is a Hilbert--Schmidt operator and 
suppose $\dmeas \in {\mathcal P}^{>}(t,\cone) \cap {\mathcal P}^\times(\ctwo)$ and $f^\dagger \in \Theta(s,R_0)$
for some constants $s,t,R_0, \cone, \ctwo>0$. Moreover, we assume that $\fmnd$ is the ML estimator defined by identity \eqref{eq:def_fmnd}.
Let $0< \eta < 1$ satisfy
\begin{equation}
	\label{eq:probability_interval}
	\log\left(\frac 8\eta\right) \leq \frac{\sqrt N}{12} \lambda_{\min}(P_mB_{\dmeas}P_m).
\end{equation}
There exists a constant $C$ depending on $D_j$, $j=1,2$ such that
\begin{equation}
	\label{eq:recon_error_prop}
	\norm{\fmnd - f^\dagger}_{{\mathcal H}} \leq C\left[R_0 m^{-s} + \log\left(\frac 8\eta\right) \cdot \delta\left(\frac{m^t}{N} + \frac{m^{\frac{t+1}2}}{\sqrt N}\right)\right]
\end{equation}
with probability greater than $1-\eta$. 
\end{theorem}

Notice that with the condition \eqref{eq:probability_interval} theorem \ref{thm:main_result_prob} characterizes the error for tail probabilities. To derive convergence results in expectation, the authors in \cite{blanchard2018optimal} utilize a priori bound for the Tikhonov regularized solution valid almost surely, which can be merged with a sharper result concerning the tail probabilities. Here, such a priori bound is not immediately available as the ML estimator can have arbitrary large norm with positive probability. To mitigate this effect, we consider a non-linear truncated estimator inspired by earlier work in \cite{cohen2013stability}.
More precisely, let $T_R\!:{\mathcal H}\to {\mathcal H}$ satisfy
\begin{equation}
\label{eq:def_trunc}
T_R(f)=\begin{cases}f&\text{if}~\|f\|\leq R,\\
0,&\text{otherwise}\end{cases}
\end{equation}
where $R\geq 0$ is fixed. We define a non-linear estimator $\est = \est(\by)$ by setting
\begin{equation}
	\label{eq:def_est}
	\est = T_{R}(\fmnd),
\end{equation}
where $R$ is chosen depending on parameters specified below.
In the following result we characterize the \emph{upper rate of convergence in $L^p$} (see \cite[Def. 3.1]{blanchard2018optimal}) for $\est$. To that end, let ${\mathcal K}(\Theta)$ be the set of regular conditional probability distributions $\rho(\cdot | \cdot) : \R\times \dom \to \R_+$ such that $\rho(\cdot | x)$ is distributed according to the Gaussian distribution ${\mathcal N}(Af^\dagger(x), \delta^2)$ for some $f^\dagger \in \Theta$.
Moreover, let us denote the admissible class of models by
\begin{equation*}
	{\mathcal M}(\Theta',  {\mathcal P}') 
	= \left\{\jmeas(dx,\, dy)  = \likemeas(dy \, | \, x) \dmeas(dx) \; \Bigg| \; 
	\likemeas(\cdot | \cdot) \in {\mathcal K}(\Theta'), \dmeas \in {\mathcal P}'\right\},
\end{equation*}
where ${\mathcal P}' \subset {\mathcal P}$ and $\Theta' \subset {\mathcal H}$ are a subset of probability measures and a source condition, respectively, specified by some fixed admissible subspaces $V_m \subset {\mathcal H}$. 

\begin{theorem}
\label{thm:main_result_exp}
Let $\{V_m\}_{m=1}^\infty$ be a sequence of admissible subspaces and suppose $s,t,R_0, \cone, \ctwo>0$ satisfy $2s-t+1>0$. Moreover, we assume that $\est$ is the ML estimator defined by identity \eqref{eq:def_est}.
For the parameter choice 
\begin{equation}
	\label{eq:opt_parameter_choice}
	m = \left(\frac{\delta}{R_0 \sqrt N}\right)^{-\frac{2}{2s+t+1}}
\end{equation}
with $R = R(m,\delta)$ given by 
\begin{equation}
	\label{eq:choice_of_R}
	R = C\left(\frac{\delta}{\lambda_{min}(P_m B_\nu P_m)}  + 1\right)
\end{equation}
with suitably large constant $C$ depending on $\ctwo$ and $R_0$,
it holds that
\begin{equation*}
	\sup_{(\delta, R_0) \in \R^2_+}  \limsup_{N\to\infty} \sup_{\mu \in {\mathcal M}'} \frac{ \left(\E \norm{\est - f^\dagger}_{{\mathcal H}}^p\right)^{\frac 1p}}{a_{N, R_0, \delta}} < \infty 
\end{equation*}
where $\bdmeas_N = \otimes_{n=1}^N \dmeas$, ${\mathcal M}' = {\mathcal M}(\Theta(s,R_0), {\mathcal P}^{>}(t,\cone) \cap {\mathcal P}^\times(\ctwo))$ and
\begin{equation*}
	a_{N, R_0, \delta} = R_0 \left(\frac{\delta}{R_0\sqrt N}\right)^{\frac{2s}{2s+t+1}}.
\end{equation*}
\end{theorem}

Following the techniques utilized in \cite{blanchard2018optimal} it is now possible to prove minimax optimality of the upper convergence rate provided above. We formulate this result as the following theorem. 

\begin{theorem}
\label{thm:minimax_optimality}
Let $s,t,R_0,\cone, \ctwo, \cthree>0$, $2s-t+1>0$, $B_\nu$ is a Hilbert--Schmidt operator and consider the design measures
\begin{equation*}
	 {\mathcal P}' = \Big\{\dmeas \in {\mathcal P} \;  \Big| \; \dmeas \in {\mathcal P}^{>}(t,\cone) \cap {\mathcal P}^\times(\ctwo) \cap {\mathcal P}^{<}(t, \cthree)\Big\} 
\end{equation*}
and $\Theta' = \Theta(s,R_0)$ given by \eqref{eq:source}.
Then the sequence of estimators $\est$ with parameter choice rule \eqref{eq:opt_parameter_choice} and \eqref{eq:choice_of_R} is strong minimax optimal in $L^p$ for all $p>0$ over the class ${\mathcal M}(\Theta', {\mathcal P}')$, i.e., the upper rate of convergence $a_{N,R_0,\delta}$ given by theorem \ref{thm:main_result_exp} is also strong minimax lower rate of convergence such that
\begin{equation*}
	\inf_{(\delta, R_0) \in \R^2_+} \liminf_{n\to\infty} \inf_{\hat f} \sup_{\mu \in {\mathcal M}(\Theta', {\mathcal P}')} \frac{ \left(\E \norm{\hat f - f^\dagger}_{{\mathcal H}}^p\right)^{\frac 1p}}{a_{N, R_0, \delta}} >0,
\end{equation*}
where the infimum is taken over all estimators (measurable mappings) $\hat f : \dom^N \times \R^N \to \mathcal H$.
\end{theorem}

\begin{remark}
\label{rem:comparison_with_blan}
There is partial overlap between results provided by Blanchard and M\"ucke in \cite{blanchard2018optimal} and those proved here. Namely,
truncated singular value decomposition is a classical method of spectral regularization,
where the unknown is projected to a finite number of eigenvectors of $B_\dmeas$. 
%A corresponding setup in our work would be to limit to considering subspaces $V_m$ spanned by the $m$ first eigenvectors of $B_\dmeas$. 
%Recall that minimax optimal rates of spectral regularization derived in \cite{blanchard2018optimal} partially overlap with our results, when considering truncated singular value decomposition.
%Indeed, the convergence rates can be compared by interpreting the main components of the result in \cite{blanchard2018optimal} in our framework. 

Suppose $B_\nu$ has a polynomially decaying spectrum and $V_m$ is spanned by the  eigenvectors corresponding to the $m$ largest eigenvalues. 
First, it is straightforward to prove that the classical source condition satisfies
\begin{equation*}
	\Xi(r,R_0) := \{f = B_\nu^{r} g \; | \; \norm{g}_{{\mathcal H}} \leq R_0\} \subset \Theta\left(rt,R_0\right).
\end{equation*}
Second, we have $\lambda_{min}(P_m B_\dmeas P_m) = \lambda_m$, where $\{\lambda_m\}_{m=1}^\infty$ are the eigenvalues of $B_\nu$ in decreasing order, giving rise to the interpretation of our parameter $t$ as the polynomial decay of spectrum of $B_\nu$ (denoted by $b$ in \cite[p. 983]{blanchard2018optimal}). Furthermore, we have
\begin{equation*}
	\norm{(P_m B_\dmeas P_m)^+ B_\dmeas (I-P_m)} = 0	
\end{equation*}
for any $m\in \N$. 
Third, more general noise model is utilized in \cite{blanchard2018optimal}. For normally distributed noise, this would correspond to $\sigma = \delta$.
Finally, we notice that the qualification of the spectral projection is arbitrary \cite[Example 2.16]{blanchard2018optimal}.

As a consequence, the rate demonstrated by Blanchard and M\"ucke is given by
\begin{equation*}
	\tilde a_{N, R_0, \sigma}= R_0 \left(\frac{\delta}{R_0\sqrt N}\right)^{\frac{2br}{2br+r+1}}
	= R_0 \left(\frac{\delta}{R_0\sqrt N}\right)^{\frac{2s}{2s+t+1}} = a_{N, R_0, \delta},
\end{equation*}
where we identified $\sigma = \delta$, $s = rt$ and $b=t$.
\end{remark}

Let us also point out an interesting difference in the respective results: In our work, the upper convergence rate requires design measure to be limited to a set of type ${\mathcal P}^>$, where the ill-posedness of the inverse problem is limited. Similarly, the minimax optimality is given in a framework further limiting design measure to a set ${\mathcal P}^<$ that restricts "well-posedness" of the problem. Noticeably the setup is exactly the opposite in \cite{blanchard2018optimal}, where upper rate requires limiting design measure to ${\mathcal P}^<$. 
We would argue that the former is more natural as intuitively worse ill-posedness should lead to deteriorated identifiability and, therefore, slower convergence rate (likewise, better conditioning should lead to improved convergence rate). 
However, the difference can be attributed to the source conditions: the classical source condition intertwines the ill-posedness of the problem with the assumption on the smoothness of the ground truth, while in our case the dependence is less direct. While our source condition is independent of the forward operator, some interplay is required due to same subspace structure being applied in the assumptions regarding the design measure.

\subsection{Literature overview}

The literature on regularization by projection or discretization is extensive; for a wide perspective on the deterministic problem in terms of projections both in domain and range of a (nonlinear) forward operator, see e.g.
\cite{natterer1977regularisierung, groetsch1988convergence, vainikko1988self, plato1990regularization,groetsch1995regularization, kaltenbacher2000regularization, bruckner2002self,  hamarik2002solution, dahmen2006error, hofmann2007regularization, mathe2008regularization,kaltenbacher2012convergence, reginska2013two,kindermann2016projection}.
For adaptive or multigrid approaches, see \cite{kaltenbacher2001regularizing, kaltenbacher2003v, maass1998adaptive}. We also mention a recent data-driven approach, where the projections are designed based on data \cite{aspri2020data}.

Projection methods have also been studied in the framework of statistical inverse problems, where minimax optimality of the method is understood in various settings
\cite{wahba1977practical, nychka1989convergence, johnstone1991discretization, donoho1995nonlinear, mair1996statistical, lukas1998comparisons, cavalier2002sharp,chow1999statistical}.
Notice that the framework provided in aforementioned papers for statistical inverse problems is connected to the learning setting in the sense that point evaluation of the data function can be considered as a projection to kernel functions on a reproducing kernel Hilbert space. However, point evaluations arising from a random measure are not covered by this theory.

In this regard, our work has close connections to the reproducing kernel methods in learning theory, which is a popular field with a vast body of literature. Let us note that connections of kernel regression methods to regularization theory were first studied in \cite{de2006discretization, vito2005learning, gerfo2008spectral} and the line of research has since become widely popular. Early work on upper rates of convergence in a reproducing kernel Hilbert space was carried out by Smale and Cucker in \cite{cucker2002best}, where they utilized a covering number technique. After the initial success, there has been a long line of subsequent work \cite{vito2005learning, smale2005shannon, smale2007learning, bauer2007regularization, yao2007early, caponnetto2007optimal} providing convergence rates comparable to \cite{blanchard2018optimal}. 
Let us also point out that there is an avenue of research \cite{mendelson2010regularization, steinwart2009optimal} considering penalties of type $R(f) = \frac 1p \norm{f}_X^p$. Notice that the notion of convergence in the usual learning context  and the inverse problem setting is different and are not directly comparable: in learning theory the convergence rates are derived in $L^2(\dmeas)$ norm, where $\dmeas$ is the unknown sampling measure generating data points. However, since the solution and data space, i.e. ${\mathcal H}$ and $H_k$, differ for the inverse problem, it is natural to consider modes of convergence in ${\mathcal H}$. For a related discussion and brief overview on relevant convergence rate literature, see \cite{Muecke2017}.

In terms of inverse learning problems, we mention that Tikhonov regularization of non-linear inverse problems is considered in \cite{rastogi2019convergence} and adaptive parameter choice rules are studied 
in \cite{lu2020balancing}. Moreover, for distributed learning of inverse problems, see \cite{guo2017learning} and references therein.

This paper is organized as follows. In section 2 we provide mathematical preliminaries and record well-known concentration results that will be utilized later. Section 3 contains a proof for theorem \ref{thm:main_result_prob}, while theorem \ref{thm:main_result_exp} is derived in section 4.  In section 5 prove that the obtained rates are minimax optimal. Finally, section 5 is we provide brief conclusions.

\section{Mathematical preliminaries}

Recall the definition of a reproducing kernel Hilbert space.

\begin{definition}
A Hilbert space $H_k$ of functions $f: \dom \to \R$ with inner product $\langle \cdot, \cdot \rangle_{H_k}$ is called the
\emph{reproducing kernel Hilbert space (RKHS)} corresponding to a symmetric, positive definite kernel $k : \dom\times \dom \to \R$ if 
\begin{itemize}
	\item[i)] for all $x\in \dom$, $k(x,x')$, as a function of its second argument, belongs to $H_k$,
	\item[ii)] for all $x \in \dom$ and $f\in H_k$, $\langle f, k(x,\cdot)\rangle_{H_k} = f(x)$.
\end{itemize}
\end{definition}

We assume below, without loss of generality, that the kernel elements are uniformly bounded $\norm{k(x,\cdot)}_{H_k} \leq 1$ and, consequently, $\norm{S_x}_{H_k \to \R}\leq 1$ for any $x\in \dom$. Moreover, recall that $k$ was assumed to be strictly positive definite. Here, we assign $\R^N$ with the empirical inner product $\langle \cdot, \cdot\rangle_N$ given by
\begin{equation*}
		\langle \by, {\bf z}\rangle_N = \frac 1N \sum_{n=1}^N y_n z_n.
\end{equation*}

Let us now consider the adjoint operator $S_X^* : \R^N \to H_k$ that satisfies
\begin{equation*}
	\langle f, S_X^* \by\rangle_{H_k} = \frac 1N\sum_{n=1}^N \langle f, k(x_n,\cdot)\rangle_{H_k} y_n 
\end{equation*}
%\begin{eqnarray*}
%	\langle f, S_X^* \by\rangle_{H_k} & = &
%	\langle S_X f, \by \rangle_{N} \\
%	& = & \frac 1N \sum_{n=1}^N f(x_n) y_n \\
%	& = & \frac 1N\sum_{n=1}^N \langle f, k(x_n,\cdot)\rangle_{H_k} y_n 
%\end{eqnarray*}
for any $f \in H_k$ and $\by = (y_n)_{n=1}^N \in \R^n$. Therefore, we have
\begin{equation*}
	S_X^* \by = \frac 1N\sum_{n=1}^N y_n k(x_n,\cdot)
\end{equation*}
and, consequently,
\begin{equation*}
	S_X^* S_X = \frac 1N \sum_{n=1}^N \langle k(x_n,\cdot),\cdot\rangle_{H_k} k(x_n,\cdot).
\end{equation*}
In similar vein, we introduce a short-hand notation $B_X$ for the normal operator
\begin{equation*}
	B_X = A^* S_X^* S_X A : H_k \to H_k.
\end{equation*}
We write $B_x$, when the ensemble $X$ is identified with one point $X = \{x\}$.
Clearly, we have $\E B_x = B_\dmeas$, when $x\sim \dmeas$. 
%Throughout this paper, we assume $\norm{A}\leq 1$, and, in particular,
%\begin{equation}
%	\label{eq:SX_norm_bound}
%	\norm{B_x} \leq 1,
%\end{equation}
%without loss of generality.

The operator $B_X$ can be considered as an empirical estimator of $B_\dmeas$ defined by identity \eqref{eq:Brho}. Its concentration rate is known and made precise by the corollary \ref{cor:op_concentration} below. To prove corollary \ref{cor:op_concentration}, the following general concentration result is used and will be utilized also elsewhere in this paper.

\begin{theorem}[{\cite[Cor. 1]{pinelis1986remarks}}]
\label{thm:inf_concentration}
Let $(\Omega, {\mathcal B}, \Prob)$ be a probability space and $\xi : \Omega \to H$ a random variable to a real separable Hilbert space $H$. Assume that there are two positive constants $L$ and $\sigma$ such that for any $p\geq 2$
\begin{equation}
	\label{eq:concentration}
	\E\norm{\xi - \E \xi}_H^p \leq \frac 12 p! \sigma^2 L^{p-2}.
\end{equation}
If the sample $\omega_1, ..., \omega_N$ is drawn i.i.d from $\Omega$ according to $\Prob$, then, for any $0<\eta<1$, we have
\begin{equation*}
	\norm{\frac 1N \sum_{j=1}^N \xi(\omega_j) - \E \xi}_H \leq \delta(N, \eta)
\end{equation*}
with probability greater than $1-\eta$, where
\begin{equation}
	\delta(N, \eta) = 2 \log\left(\frac 2\eta\right)\left(\frac L N + \frac \sigma{\sqrt N}\right).
\end{equation}
In particular, inequality \eqref{eq:concentration} holds if
\begin{eqnarray*}
	\norm{\xi(\omega)}_H & \leq &  \frac L 2\quad \text{a.s.} \quad \text{and} \\
	\E \norm{\xi}_H^2 & \leq & \sigma^2.
\end{eqnarray*}
\end{theorem}

Let us briefly note that a normal random variable $\xi \in {\mathcal N}(0,\delta^2)$ satisfies \eqref{eq:concentration} with $\sigma = L = \delta$.

\begin{corollary}[{\cite[Prop. 5.5]{blanchard2018optimal}}]
\label{cor:op_concentration}
Let $B_\nu$ be a Hilbert--Schmidt operator and $\norm{A}\leq 1$.
For any sample size $N>0$ and $0<\eta<1$ it holds that
\begin{equation*}
	\norm{B_\dmeas - B_X}_\HS \leq 6 \log\left(\frac 2\eta\right) \frac 1{\sqrt N}
\end{equation*}
with probability greater than $1-\eta$.
\end{corollary}

%Recall that the Hilbert--Schmidt norm dominates the usual operator norm.

In what follows, $B^+$ stands for the Moore--Penrose pseudoinverse of an operator $B$ \cite{engl1996regularization}.
Notice that we frequently utilize the following two properties of pseudoinverse
\begin{equation}
\label{eq:aux1}
	(P_m B P_m)^+ =P_m(P_m B P_m)^+=(P_m B P_m)^+ P_m
\end{equation}
and 
\begin{equation}
	\label{eq:aux2}
	(P_m B P_m)^+ P_m B P_m=P_m B P_m (P_m B P_m)^+,
\end{equation}
where $B : {\mathcal H} \to {\mathcal H}$ is a linear bounded operator. 

\section{Concentration result}

In this section we assume that $B_\nu$ is a Hilbert--Schmidt operator and derive the proof for theorem \ref{thm:main_result_prob}. To this end, 
we will make gradual assumptions on the interplay between $B_\dmeas$ and subspaces $V_m$
to build towards the framework $\dmeas \in  {\mathcal P}^{>}(t,\cone) \cap {\mathcal P}^\times(\ctwo)$ assumed in the theorem.
For convenience, let us abbreviate
\begin{equation*}
	\lambda_m = \lambda_{min}(P_m B_\dmeas P_m)
\end{equation*}
for any $m\geq 1$.
Later, we provide the precise connection of the final estimate to constants $\cone$ and $\ctwo$ and, therefore, keep track of them below.

Next, consider the definition of $\fmnd$. A least squares solution $f$ to $S_X A f = \byd$ in $V_m$ clearly solves the normal equation
\begin{equation*}
	P_m B_X P_m f - (S_X A P_m)^* \byd = 0 \quad \text{a.s.}
\end{equation*}
It follows that $\fmnd$ is obtained with the pseudoinverse 
\begin{equation*}
	\fmnd = \left(P_m B_X P_m\right)^+ (S_X A P_m)^* \byd \quad \text{a.s.},
\end{equation*}
and, consequently,
\begin{eqnarray}
	\fmnd - f^\dagger & = & \left(P_m B_X P_m\right)^+ (S_X A)^* \byd - f^\dagger \nonumber \\
	& = & \left(\left(P_m B_X P_m\right)^+ B_X - I \right)f^\dagger + \delta \left(P_m B_X P_m\right)^+ (S_X A)^* \be \nonumber \\
	& =: & I_1 + I_2, \label{eq:recon_err_id}
\end{eqnarray}
where we utilized identity \eqref{eq:aux1} for $B=B_X$ and abbreviated
\begin{equation}
	\label{eq:idforI1I2}
	I_1 = \left(\left(P_m B_X P_m\right)^+ B_X - I \right)f^\dagger
	\quad \text{and} \quad 
	I_2 = \delta \left(P_m B_X P_m\right)^+ (S_X A)^* \be.
\end{equation}
In literature, the term $I_1$ is often called \emph{approximation error}, while $I_2$ is referred to as the \emph{variance}.

Let us derive some concentration properties for $P_m B_X P_m$ utilizing corollary \ref{cor:op_concentration}. Carefully notice that below we vary the operator norm between the standard norm $\norm{\cdot}$ of linear bounded operators and the Hilbert--Schmidt norm $\norm{\cdot}_\HS$. Recall that $\norm{B} \leq \norm{B}_\HS$ and
\begin{equation*}
	\norm{A_1 B A_2}_\HS \leq \norm{A_1} \norm{B}_\HS \norm{A_2}
\end{equation*}
for any linear bounded operators $A_1, A_2, B \in {\mathcal L}({\mathcal H})$.

\begin{lemma}
\label{lem:concentration_prob}
Let $0< \eta < 1$ satisfy
\begin{equation}
	\label{eq:firstetabound}
	\log\left(\frac 2\eta\right) \leq \frac {\sqrt N}{12} \lambda_m.
\end{equation}
With probability greater than $1-\eta$, it holds that
\begin{equation}
\label{eq:concentrationassumption}
	\norm{B_X- B_\dmeas}_\HS \leq\frac 12 \lambda_m.
\end{equation}
\end{lemma}

\begin{proof}
Reformulating inequality \eqref{eq:firstetabound} we observe that
\begin{equation*}
	6\log\bigg(\frac{2}{\eta}\bigg)\frac{1}{\sqrt N} \leq \frac12 \lambda_m.
\end{equation*}
Now corollary~\ref{cor:op_concentration} immediately yields the result.
\end{proof}

We note that since for any $\eta$ satisfying $\eta\geq 2\,\exp\left(-\frac{1}{12}\lambda_m\sqrt{N}\right)$ and $\eta<1$, we also have $\lambda_m\sqrt{N}> 12\log 2$.

\begin{proposition}\label{prop:inversebounds}
Suppose $\dmeas \in {\mathcal P}^>(t,D_1)\cap {\mathcal P}^\times(\ctwo)$ and
inequality \eqref{eq:concentrationassumption} holds.
Then it follows that
\begin{eqnarray}
	\norm{P_m - (P_m B_\dmeas P_m)^+ P_m B_X P_m}_\HS & \leq & \frac 12,  \label{eq:conprop1_a1}\\
	\norm{(P_m B_X P_m)^+} & \leq & \frac{2}{\lambda_m},  \label{eq:conprop1_a2}\\
	\norm{(P_m B_X P_m)^+(P_m B_\dmeas P_m)} & \leq & 2(1+\ctwo) \quad\quad\quad\quad \text{and}\label{eq:conprop1_a3}\\
	\norm{(P_m B_X P_m)^+ B_X (I-P_m)} & \leq & 2\ctwo + 4. \label{eq:conprop1_a4}
\end{eqnarray}
\end{proposition}

\begin{proof}
First observe that $P_m B_\dmeas P_m$ is a bijection from $V_m$ onto itself, since $\nu \in {\mathcal P}^>(t,D_1)$, and by inequality \eqref{eq:concentrationassumption} we have $\norm{P_m(B_X - B_\dmeas)P_m}\norm{(P_m B_\dmeas P_m)^+}\leq \frac 12< 1$. Therefore, by Neumann series theorem also $P_m B_X P_m = P_m(B_X - B_\dmeas)P_m + P_m B_\dmeas P_m$ is bijection on $V_m$ and we have 
\begin{equation*}
	(P_m B_X P_m)^+ P_m B_X P_m = P_m B_X P_m (P_m B_X P_m)^+ = P_m.
\end{equation*}

For inequality \eqref{eq:conprop1_a1} we notice that
\begin{equation}
	\label{eq:aux3}
	\norm{P_m - (P_m B_\dmeas P_m)^+ P_m B_X P_m}_\HS 
	\leq \norm{(P_m B_\dmeas P_m)^+} \norm{P_m (B_\dmeas - B_X) P_m}_\HS \leq \frac 12,
\end{equation}
by our assumptions, where we also applied~\eqref{eq:aux2}.

For remaining inequalities, we find by some algebraic manipulation that
\begin{eqnarray}
	(P_m B_X P_m)^+ & = & (P_m B_X P_m)^+ - (P_m B_\dmeas P_m)^+ + (P_m B_\dmeas P_m)^+ \nonumber \\
	& = & (P_m - (P_m B_\dmeas P_m)^+ P_m B_X P_m) (P_m B_X P_m)^+ + (P_m B_\dmeas P_m)^+. \label{eq:con_aux1}
\end{eqnarray}
First, by triangle inequality we obtain
\begin{eqnarray*}
\norm{(P_mB_XP_m)^\dag} & \leq & \|P_m-(P_mB_\dmeas P_m)^\dag P_mB_XP_m\|\|(P_mB_XP_m)^\dag\|+\|(P_mB_\dmeas P_m)^\dag\| \\
& \leq & \|(P_mB_\dmeas P_m)^\dag\|\|P_m (B_\dmeas - B_X) P_m\|_\HS\|(P_mB_XP_m)^\dag\|+\|(P_mB_\dmeas P_m)^\dag\|.
\end{eqnarray*}
and since $\|(P_mB_\dmeas P_m)^\dag\| = 1/\lambda_m$, it follows that
\begin{eqnarray*}
	\norm{(P_m B_X P_m)^\dag} \leq \frac 12 \norm{(P_m B_X P_m)^\dag} + \frac 1{\lambda_{\min}}.
\end{eqnarray*}
This yields inequality \eqref{eq:conprop1_a2}.

Second, we multiply the identity \eqref{eq:con_aux1} from right by $B_\dmeas$ and apply triangle inequality together with a norm bounds so that
\begin{multline*}
	\norm{(P_m B_X P_m)^+ B_\dmeas} \\
	\leq \norm{P_m - (P_m B_\dmeas P_m)^+ P_m B_X P_m}_\HS \norm{(P_m B_X P_m)^+ B_\dmeas} + \norm{(P_m B_\dmeas P_m)^+ B_\dmeas} \\
	\leq \frac 12 \norm{(P_m B_X P_m)^+ B_\dmeas} + \norm{(P_m B_\dmeas P_m)^+ B_\dmeas},
\end{multline*}
where we applied the inequality \eqref{eq:conprop1_a1}. Now inequality \eqref{eq:conprop1_a3} follows by noting that 
\begin{equation*}
	\norm{(P_m B_\dmeas P_m)^+ B_\dmeas} \leq 1 + \ctwo
\end{equation*}
due to identity \eqref{eq:aux1} and our assumption $\dmeas \in {\mathcal P}^\times(\ctwo)$.

In the same vein, multiplying identity \eqref{eq:con_aux1} from right by $B_X (I-P_m)$ yields
\begin{align*}
	\|(P_m B_X & P_m)^+  B_X (I-P_m)\| \\
	& \leq \norm{P_m - (P_m B_\dmeas P_m)^+ P_m B_X P_m}_\HS \norm{(P_m B_X P_m)^+ B_X (I-P_m)} \\ & \quad \quad + \norm{(P_m B_\dmeas P_m)^+ B_X (I-P_m)} \\
	& \leq \frac 12 \norm{(P_m B_X P_m)^+ B_X (I-P_m)} + \norm{(P_m B_\dmeas P_m)^+ B_X (I-P_m)},
\end{align*}
and, consequently,
\begin{equation*}
	\norm{(P_m B_X P_m)^+ B_X (I-P_m)} \leq 2 \norm{(P_m B_\dmeas P_m)^+ B_X (I-P_m)}.
\end{equation*}
Now we have
\begin{multline*}
	\norm{(P_m B_\dmeas P_m)^+ B_X (I-P_m)} \\
	\leq \norm{(P_m B_\dmeas P_m)^+ B_\dmeas (I-P_m)}
	+ \norm{(P_m B_\dmeas P_m)^+} \norm{B_X-B_\dmeas}_\HS
	\leq \ctwo + 2,
\end{multline*}
which concludes the proof.
\end{proof}
%
%\vk{\emph{Remark}. Suppose that
%\begin{equation*}
%	\norm{B_X- B_\dmeas}_\HS \leq \delta,\quad \delta\in (0,\lambda_{\min}(P_mB_\dmeas P_m)).
%\end{equation*}
%Then the following four inequalities hold:
%\begin{eqnarray*}
%	\|(P_m B_X P_m)^+\| & \leq & \frac{1}{\lambda_{\min}(P_mB_\dmeas P_m)-\delta}, \\
%	\norm{P_m - (P_m B_\dmeas P_m)^+ P_m B_X P_m} & \leq & \frac{\delta}{\lambda_{\min}(P_mB_\dmeas P_m)}, \\
%	\norm{(P_m B_X P_m)^+(P_m B_\dmeas P_m)} & \leq & \frac{\norm{(P_m B_\dmeas P_m)^+ B_\dmeas}}{1-\frac{\delta}{\lambda_{\min}(P_mB_\dmeas P_m)}} \quad \text{and}\\
%	\norm{(P_m B_X P_m)^+ B_X (I-P_X)}_{\mathcal L({\mathcal H})} & \leq & C+\frac{\delta}{\lambda_{\min}(P_mB_\dmeas P_m)}.
%	\end{eqnarray*}}

\begin{theorem}
\label{thm:appr}
Suppose that $\dmeas \in {\mathcal P}^>(t,D_1) \cap {\mathcal P}^\times(\ctwo)$ and $f^\dag \in \Theta(s,R_0)$. 
For any $m\in \N$, let $0< \eta < 1$ satisfy
\begin{equation*}
	\log\left(\frac 8\eta\right) \leq \frac {\sqrt N}{12} \lambda_m.
\end{equation*}
Then it holds that
\begin{equation}
	\label{eq:I1_bound}
	\norm{I_1}_{\mathcal H} \leq CR_0m^{-s}
\end{equation}
for $C = 2\ctwo+5$ holds with probability greater than $1-\eta/4$, where $I_1$ is given in equation \eqref{eq:idforI1I2}.
\end{theorem}

\begin{proof}
Let us first observe that
\begin{equation*}
	I_1 = \left(\left(P_m B_X P_m\right)^+ B_X - P_m \right)f^\dagger + (P_m - I)f^\dagger.
\end{equation*}
When the bound \eqref{eq:concentrationassumption} holds, we have that
\begin{eqnarray*}
	\left(\left(P_m B_X P_m\right)^+ B_X - P_m \right)f^\dagger
	& = & \left(P_m B_X P_m\right)^+ (B_X - P_m B_X P_m) f^\dagger \\
	 & = & \left(P_m B_X P_m\right)^+ B_X\left(I- P_m\right)f^\dagger,
\end{eqnarray*}
where we used identities \eqref{eq:aux1} and~\eqref{eq:aux3}.
Now combining inequality \eqref{eq:conprop1_a4} in proposition \ref{prop:inversebounds} with the source condition implies
\begin{equation*}
	\norm{I_1}_{{\mathcal H}} \leq \left(\norm{\left(P_m B_X P_m\right)^+ B_X\left(I- P_m\right)}+1\right) \norm{\left(I- P_m\right)f^\dagger}_{{\mathcal H}} \leq (2\ctwo + 5) R_0 m^{-s}
\end{equation*}
with the given probability.
\end{proof}

In what follows, we abuse notation by denoting the square root of the pseudoinverse $(P_m B_X P_m)^{+}$ by $(P_m B_X P_m)^{-\frac 12}$ for convenience.

\begin{theorem}
\label{thm:variance}
Suppose $\dmeas \in {\mathcal P}^>(t,\cone) \cap {\mathcal P}^\times(\ctwo)$. Let $0< \eta < 1$ satisfy
\begin{equation}
	\label{eq:var_eta_req}
	\log\left(\frac 8\eta\right) \leq \frac {\sqrt N}{12} \lambda_m.
\end{equation}
There exists a constant $C$ dependent on $\cone$ and $\ctwo$ such that 
\begin{equation*}
	\norm{I_2}_{\mathcal H} \leq C\delta \log\left(\frac 8\eta\right)\left(\frac{m^t}{N} + \frac{m^{\frac{t+1}2}}{\sqrt N}\right),
\end{equation*}
with probability greater than $1-\frac 34 \eta$, where $\delta>0$ is the noise level.
\end{theorem}

\begin{proof}
Let us decompose $I_2$ into three terms
\begin{equation*}
	I_2 = \delta \underbrace{(P_m B_X P_m)^{-\frac 12}}_{=: K_1} \cdot \underbrace{(P_m B_X P_m)^{-\frac 12} (P_m B_\dmeas P_m)^{\frac 12}}_{=: K_2} \cdot \underbrace{(P_m B_\dmeas P_m)^{-\frac 12} A^* S_X^* \be}_{=: K_3}.
\end{equation*}
Below we use the observation that for $\eta$ satisfying \eqref{eq:var_eta_req} we have by lemma \ref{lem:concentration_prob} that $\norm{B_X - B_\nu}_{HS} \leq \lambda_m/2$ with probability greater than $1-\eta/4$ and, therefore, the inequalities of proposition \ref{prop:inversebounds} will be available with the same probability.

First, under the assumption of inequality \eqref{eq:concentrationassumption}, we obtain by inequality \eqref{eq:conprop1_a2} that
\begin{equation*}
	\norm{K_1}=\norm{(P_mB_XP_m)^{-1/2}}=\sqrt{\norm{(P_mB_XP_m)^+}} \leq \sqrt{\frac{2}{\lambda_m}}
	\leq \sqrt {\frac 2{\cone}} m^{\frac t2}
\end{equation*}
with probability greater than $1-\eta/4$, where we used the assumption that $\dmeas \in {\mathcal P}^<(t,\cone)$.

For the second term we find that
\begin{equation*}
	\norm{K_2} = \norm{(P_m B_X P_m)^{-\frac 12} (P_m B_\dmeas P_m)^{\frac 12}} 
	\leq \sqrt{\norm{(P_m B_X P_m)^+ (P_m B_\dmeas P_m)}}
	\leq \sqrt{2(1+\ctwo)}
\end{equation*}
with probability greater than $1-\eta/4$, where we applied the Cordes inequality \cite[theorem IX.2.1-2]{bhatia2013matrix} and inequality \eqref{eq:conprop1_a3} from proposition \ref{prop:inversebounds}.

For the third term, let us write
\begin{equation*}
	\xi_n  = (P_m B_\dmeas P_m)^{-\frac 12} A^* S_{x_n}^* \epsilon_n.
\end{equation*}
Consequently, we have identity
\begin{equation*}
	K_3 = \frac 1N \sum_{n=1}^N \xi_n.
\end{equation*}
Notice that $\xi_n$, $n=1,...,N$, are independent ${\mathcal H}$-valued random variables. Recall that $\E \xi_n = 0$ and
\begin{eqnarray*}
	\E \norm{\xi_n}_{\mathcal H}^p & = & \E |\langle (P_m B_\dmeas P_m)^+ A^* S_{x_n}^* \epsilon_n, A^* S_{x_n}^*\epsilon_n  \rangle_{\mathcal H}|^\frac p2 \\
	& = & \E |(S_{x_n} A(P_m B_\dmeas P_m)^+ A^* S_{x_n}^* \epsilon_n)\cdot \epsilon_n  \rangle|^\frac p2 \\
	& \leq & \E |S_{x_n} A(P_m B_\dmeas P_m)^+ A^* S_{x_n}^*|^{\frac p2} \E |\epsilon_n|^p \\
	& \leq & \frac 12 p!\, \E |S_{x_n} A(P_m B_\dmeas P_m)^+ A^* S_{x_n}^*|^{\frac p2},
\end{eqnarray*}
where we identified the linear operator $S_{x_n} A(P_m B_\dmeas P_m)^+ A^* S_{x_n}^*: \R\to\R$ with a real number and utilized a rough upper bound $\E|\epsilon_n|^p \leq p!/2$.
Due to boundedness of $A$ we observe that 
\begin{equation*}
	|S_{x_n} A(P_m B_\dmeas P_m)^+ A^* S_{x_n}^*| \leq \norm{(P_m B_\dmeas P_m)^+} \leq C m^t,
\end{equation*}
where $C = D_1^{-1}$.
Moreover, abbreviating $Z_n = S_{x_n} A (P_m B_\dmeas P_m)^{-\frac 12} : {\mathcal H} \to \R$ we obtain 
\begin{equation*}
	|S_{x_n} A(P_m B_\dmeas P_m)^+ A^* S_{x_n}^*|^{\frac p2} = |Z_n Z_n^*|^{\frac p2} \leq (C m^t)^{\frac p2-1}\cdot Z_n Z_n^*
\end{equation*}
due to positivity of the operator $Z_nZ_n^*$.
In consequence, we obtain
\begin{equation*}
	\E |S_{x_n} A(P_m B_\dmeas P_m)^+ A^* S_{x_n}^*|^{\frac p2} \leq C^{\frac p2-1} m^{t(\frac p2 -1)} \tr_{\mathcal H}((P_m B_\dmeas P_m)^+ B_\dmeas) = C^{\frac p2-1}  m^{t(\frac p2 -1)+1},
\end{equation*}
since $Z_nZ_n^* = \tr_{\mathcal H}(Z_n^* Z_n)$ and $\E S_{x_n}^* S_{x_n} = \E B_{x_n} = B_\dmeas$. Moreover, we have
\begin{equation*}
	\tr_{{\mathcal H}}((P_m B_\dmeas P_m)^+ B_\dmeas) = \tr_{{\mathcal H}}(P_m) = m
\end{equation*}
due to \eqref{eq:aux2} and the cyclic property of the trace.
To sum up, we obtain
\begin{equation*}
	\E \norm{\xi_n}_{\mathcal H}^p \leq \frac 12 p!\cdot C^{\frac p2-1}  m^{t(\frac p2 -1)+1}
	=  \frac 12 p! \left(C^{\frac p4-\frac 12}  \sqrt{m}\right)^2 (m^{\frac t2})^{p-2}
\end{equation*}
for any $p\geq 2$.

Applying theorem \ref{thm:inf_concentration} yields that
\begin{equation*}
	\norm{K_3}_{\mathcal H} \leq 2\log\left(\frac 8\eta\right)\left(\frac{m^{\frac t2}}{N} + D_1^{\frac 12 - \frac p4}\sqrt{\frac{m}{N}}\right)
\end{equation*}
with probability greater than $1-\eta/4$. Combining the error estimates for $K_i$, $i=1,2,3$, we obtain the result.
\end{proof}

{\bf Proof of theorem \ref{thm:main_result_prob}.}
The result follows by combining theorems \ref{thm:appr} and \ref{thm:variance}. Namely, if we have independent events $E_1$ and $E_2$ occuring with probability greater than $1-\frac \eta 4$ and $1-\frac{3\eta}4$, respectively, then $E_1$ and $E_2$ occur simultaneously with probability
\begin{equation*}
	{\Prob}(E_1 \cap E_2) = \left(1 -\frac \eta 4\right)\left(1-\frac{3\eta}4\right) = 1 - \eta + \frac{3\eta^2}{16} \geq 1- \eta.
\end{equation*}
This concludes the proof.
\hfill\qed

\section{Expected reconstruction error}

%\subsection{Concentration result}

Recall the definition of $T_R\!:{\mathcal H}\to {\mathcal H}$ given in \eqref{eq:def_trunc}
and by assumption $f^\dag \in \Theta(s, R_0)$ it also holds that $\|f^\dag\|_{{\mathcal H}}\leq R_0$. We define our nonlinear estimator according to
\begin{equation*}
	\est = T_{R}(\fmnd),
\end{equation*}
where $R$ is set below and will depend on $m$, $\delta$ and $R_0$. 

In the following, consider $(\Omega, \Prob)$ as a complete probability space describing all possible events $\omega$ of our learning problem \eqref{eq:main_model}. We factorize $\Omega$ into smaller subsets that can be individually quantified.
Let us denote 
\begin{equation*}
	\Omega_R = \{\omega\in \Omega : \|\fmnd\|_{{\mathcal H}} \leq R\}.	
\end{equation*}
Moreover, let 
\begin{equation*}
	\Omega_+:=\left\{\omega\in \Omega:\|B_X-B_{\dmeas}\|_{\rm HS}\leq \frac12\lambda_m\right\}
\end{equation*}
and $\Omega_-:=\Omega\setminus\Omega_+$. Notice that by lemma \ref{lem:concentration_prob} we have
\begin{equation}
	\label{eq:third_term}
	\Prob(\Omega_-) \leq 2\,\exp\left(-\frac{\sqrt{N}}{12}\lambda_m\right).
\end{equation}

Utilizing inequality 
\begin{equation*}
	\|f^\dag-\est\|_{{\mathcal H}}^p\leq 2^{p-1}\left(\|f^\dag\|_{{\mathcal H}}^p+\|\est\|_{{\mathcal H}}^p\right)
\end{equation*}
for any $p\geq 1$ and the trivial inequality $\|\est\|\leq R$, we can decompose the expected reconstruction error as follows
\begin{align}
\label{eq:error_decomp}
\E \|f^\dag & - \est \|_{{\mathcal H}}^p \nonumber \\
& \leq  \int_{\Omega_+\cap \Omega_R}\|f^\dag-\fmnd \|_{{\mathcal H}}^p\,\Prob({\rm d}\omega) + R_0^p\Prob\left(\Omega_+\cap \Omega_R^c\right)+ 2^{p-1}(R_0^p + R^p) \Prob(\Omega_-) 
\end{align}
where in both second and third term we utilized the truncation.
A direct estimate to the third term in error decomposition is given by \eqref{eq:third_term}.
Let us now derive estimates also for the first and second error term in \eqref{eq:error_decomp}.

\begin{proposition}
\label{prop:J1}
Consider the model ${\mathcal M}(s, R_0, {\mathcal P}^{>}(t,\cone) \cap {\mathcal P}^\times(\ctwo))$ for some constants $s,t,R_0, \cone, \ctwo>0$, where $f^\dagger \in \Theta(s,R_0)$, and suppose that $\est$ is the ML estimator defined by identity \eqref{eq:def_est}.
We have
\begin{align*}
\int_{\Omega_+\cap \Omega_R} & \|f^\dag-\fmnd \|_{{\mathcal H}}^p\,\Prob({\rm d}\omega) \\
& \leq C\left[R_0^p m^{-ps} + \delta^p \left(\frac{m^{pt}}{N^p} + \frac{ m^{\frac{p(t+1)}{2}}}{N^{\frac p2}}\right) + (R^p+R_0^p)\exp\left(-\frac{\ctwo}{12} \sqrt N m^{-t}\right)\right]
\end{align*}
for $p \leq \sqrt{N} \lambda_m/24 -1/2$, where the constant $C$ depending on $p$ and $D_j$, $j=1,2$.
\end{proposition}
\begin{proof}
Define a positive random variable
\begin{equation*}
	X = \mathbf 1_{\Omega_+\cap \Omega_R} (\omega)\|f^\dag - \fmnd\|_{{\mathcal H}}.
\end{equation*}
Next, recall from theorem \ref{thm:main_result_prob} that for $0<\eta<1$ satisfying
\begin{equation*}
	\eta \geq 8 \exp\left(-\frac{\sqrt N}{12} \lambda_{min}\right) =: \eta_0
\end{equation*}
we have
\begin{equation*}
	\Prob\left(X \geq a + \log\left(\frac 1\eta\right) b\right) \leq \Prob\left(\|f^\dag - \fmnd\|_{{\mathcal H}}\geq a + \log\left(\frac 1\eta\right) b\right) \leq \eta,
\end{equation*}
where 
\begin{equation*}
	a = C\left[R_0 m^{-s} + \log(8) \delta \left(\frac{m^t}{N} + \frac{ m^{\frac{t+1}{2}}}{\sqrt N}\right)\right]
\end{equation*}
and 
\begin{equation*}
	b = C\delta \left(\frac{ m^t}{N} + \frac{ m^{\frac{t+1}{2}}}{\sqrt N}\right)
\end{equation*}
for some constant $C>0$ depending on $D_j$, $j=1,2$. Similarly, we observe that
\begin{equation*}
	\Prob(X \geq R+R_0) = 0
\end{equation*}
due to the source condition and requirement  $\|\fmnd\|\leq R$.

Now it follows directly from \cite[Cor. C.2.]{blanchard2018optimal} that for any positive $p\leq -\frac 12 \log \eta_0$ we have
\begin{equation*}
	\E X^p \leq C_p \left(a^p + b^p + (R+R_0)^p \eta_0 \right)
\end{equation*}
with a constant $C_p$ depending on $p$. The result follows by estimates of type $(x+y)^p \leq 2^{p-1} (x^p + y^p)$ for $x,y>0$.
\end{proof}

\begin{proposition}
\label{prop:J2}
Let $\{V_m\}_{m=1}^\infty$ be a sequence of admissible subspaces and 
suppose $\dmeas \in {\mathcal P}^{>}(t,\cone) \cap {\mathcal P}^\times(\ctwo)$ and $f^\dagger \in \Theta(s,R_0)$
for some constants $s,t,R_0, \cone, \ctwo>0$. 
Moreover, let $\est$ be the ML estimator defined by identity \eqref{eq:def_est}.
For the parameter choice 
\begin{equation}
	\label{eq:R_choice}
	R = \frac{2 \sqrt 2 \delta}{\lambda_m}  + (2\ctwo + 6)R_0
\end{equation}
it holds that
\begin{equation*}
	\Prob\left(\Omega_+ \cap \Omega_R^c\right) \leq \exp\left(-\frac N8\right).
\end{equation*}
\end{proposition}
\begin{proof}
We observe first that
\begin{eqnarray*}
\Prob\left(\Omega_+ \cap \Omega_R^c\right)
& \leq & \Prob\left(\Omega_+ \cap \{\|\fmnd-f^\dag\|+\|f^\dag\|>R\}\right) \\
& \leq & \Prob\left(\Omega_+ \cap \{\|\fmnd-f^\dag\|>R-R_0\}\right),
\end{eqnarray*}
where we made use of $\|f^\dag\|_{{\mathcal H}}\leq R_0$.

For $\omega \in \Omega_+$, we notice by identity \eqref{eq:recon_err_id} and proposition \ref{thm:appr} that
\begin{equation*}
	\norm{\fmnd - f^\dag}_{{\mathcal H}} \leq (2\ctwo+5)R_0 m^{-s} + I_2,
\end{equation*}
and, consequently,
\begin{eqnarray*}
	\Prob\left(\Omega_+ \cap \Omega_R^c\right) 
	& \leq & \Prob\left(\Omega_+ \cap \{\|I_2\|_{{\mathcal H}}>R-(2\ctwo+6)m^{-s})R_0\}\right) \\ 
	& \leq & \Prob\left(\Omega_+ \cap \left\{\|I_2\|_{{\mathcal H}}>\frac{2 \sqrt 2\delta}{\lambda_m} \right\}\right)
\end{eqnarray*}
by the choice of $R$ in equation \eqref{eq:R_choice}.
Note that in $\Omega_+$, we have that
\begin{equation*}
	\|I_2\|\leq \delta\|(P_mB_XP_m)^{-1}\|\cdot\|(S_XA)^*\be\|\leq  \frac{2 \delta}{\lambda_m} \cdot \|\be\|_N.
\end{equation*}
Then, it follows that
\begin{eqnarray*}
	\Prob\left(\Omega_+ \cap \Omega_R^c\right) & \leq & 
	\Prob\left\{ \frac{2 \delta}{\lambda_m} \cdot \|\be\|_N > \frac{2 \sqrt 2 \delta}{\lambda_m} \right\} \\
	&\leq & \Prob\left\{\norm{\be}_N > \sqrt{2}\right\} \\
%	& = & \Prob\left\{ \frac 1N \sum_{j=1}^N \epsilon_j^2 > 2\right\} \\
   & = & \Prob\left\{ \frac 1N \sum_{j=1}^N (\epsilon_j^2 - \E \epsilon_j^2) > 1\right\},
\end{eqnarray*}
where we applied identity $\E \epsilon_j^2 = 1$.
Indeed, random variables $\epsilon_j^2$ have subexponential distribution, more precisely, $\epsilon_j^2\sim {\rm SE}(\nu^2,\alpha)$ with $\nu=2$, $\alpha=4$. It follows that
$$
\sum_{j=1}^N(\epsilon_j^2-\E\epsilon_j^2)\sim {\rm SE}(4N,4),
$$
see e.g. \cite{vershynin2018high}.
By Bernstein's inequality, a random variable $X\sim{\rm SE}(\nu^2,\alpha)$ satisfies the one-sided tail bound
$$
\Prob(X-\E X>t)\leq \begin{cases}\exp\left(-\frac{t^2}{2\nu^2}\right)&\text{if}~0\leq t\leq \frac{\nu^2}{\alpha},\\
\exp\left(-\frac{t}{2\alpha}\right)&\text{if}~t>\frac{\nu^2}{\alpha}.\end{cases}
$$
By considering $X= \sum_{j=1}^N(\epsilon_j^2-\E\epsilon_j^2)$, we observe that $\nu^2/\alpha = N$ and taking $t=N$, we conclude that
\begin{equation*}
	\Prob\left(\Omega_+ \cap \Omega_R^c\right) \leq \exp\left(-\frac N8\right).
\end{equation*}
which completes the proof.

%
%\begin{align*}
%&\Prob\bigg\{\sum_{j=1}^N (\epsilon_j^2-\E[\epsilon_j^2])>\bigg(C'\frac{\lambda_{\min}(L-Cm^{-s})^2}{\delta^2}-1\bigg)N\bigg\}\\
%&\leq \begin{cases}\exp\bigg(-\frac{\big(C'\frac{\lambda_{\min}(L-Cm^{-s})^2}{\delta^2}-1\big)^2N}{8}\bigg)&\text{if}~\todo{0\leq \big(C'\frac{\lambda_{\min}(L-Cm^{-s})^2}{\delta^2}-1\big)N\leq N},\\
%\exp\bigg(-\frac{\big(C'\frac{\lambda_{\min}(L-Cm^{-s})^2}{\delta^2}-1\big)N}{8}\bigg)&\text{if}~\todo{\big(C'\frac{\lambda_{\min}(L-Cm^{-s})^2}{\delta^2}-1\big)N>N}.\end{cases}%&\leq \exp\bigg(-\frac{N}{8}\min\bigg\{\frac{\big(C'\frac{\lambda_{\min}(L-Cm^{-s})^2}{\delta^2}-1\big)^2}{N},C'\frac{\lambda_{\min}(L-Cm^{-s})^2}{\delta^2}-1\bigg\}\bigg).
%\end{align*}

\end{proof}

Combining propositions \ref{prop:J1} and \ref{prop:J2} with inequality \eqref{eq:error_decomp}, we can state the following upper bound to the reconstruction error

\begin{corollary}
\label{cor:exp_upper_bound}
Let $\{V_m\}_{m=1}^\infty$ be a sequence of admissible subspaces and 
suppose $\dmeas \in {\mathcal P}^{>}(t,\cone) \cap {\mathcal P}^\times(\ctwo)$ and $f^\dagger \in \Theta(s,R_0)$
for some constants $s,t,R_0, \cone, \ctwo>0$. 
Moreover, let $\est$ be the ML estimator defined by identity \eqref{eq:def_est} with $R=R(m,\delta)$ given by \eqref{eq:R_choice}.
We have that
\begin{multline}
	\label{eq:Lp_upper_bound}
	\E \|f^\dag - \est \|_{{\mathcal H}}^p 
	\lesssim R_0^p m^{-ps} + \frac{\delta^p m^{pt}}{N^p} + \frac{\delta^p m^{\frac{p(t+1)}2}}{N^{\frac p2}} \\
	+ (\delta^p m^{pt} + R_0^p)\left[\exp\left(-\ctwo \sqrt N m^{-t}\right) + \exp\left(-\frac N 8\right)\right],
\end{multline}
for $p\leq \sqrt{N}\lambda_m/24 - 1/2$, where the inequality is up to a constant depending on $p$ and $D_j$, $j=1,2$.
\end{corollary}

Now we are ready to prove our main result.

{\bf Proof of theorem \ref{thm:main_result_exp}.}
Let us consider the upper bound obtained in corollary \ref{cor:exp_upper_bound} and 
focus on the sum 
\begin{equation}
	L(m) = R_0^pm^{-ps} + \frac{\delta^p m^{\frac{p(t+1)}2}}{N^{\frac p2}}.
\end{equation}
An ansatz $m\propto \delta^{\nu_1} R_0^{\nu_2} N^{\nu_3}$ with variables $\nu_j \in \R$, $j=1,2,3$ minimizes $L(m)$ with parameters
\begin{equation}
	\label{eq:params_nu}
	\nu_1 = -\frac{2}{2s+t+1}, \quad  \nu_2 = \frac{2}{2s+t+1}\quad \text{and} \quad \nu_3 = \frac{1}{2s+t+1}
\end{equation}
and induces a bound
\begin{equation}
	\label{eq:L1bound}
	L(m) \lesssim R_0^p \left(\frac{\delta}{R_0\sqrt N}\right)^{\frac{2ps}{2s+t+1}} =: a_{N, R_0, \delta}^p.
\end{equation}
We will next confirm that applying the ansatz with values in \eqref{eq:params_nu} to the other terms in the upper bound of \eqref{eq:Lp_upper_bound} yields slower rates of convergence.

First, applying the ansatz to the second term, we observe that 
\begin{equation*}
	\frac{\delta^pm^{pt}}{N^p} \leq C_{\delta, R_0} N^{-\frac{(2s+1)p}{2s+t+1}}
\end{equation*}
%implying that this term is dominated by the bound \eqref{eq:L1bound} asymptotically w.r.t. $N$. As a side note, we can observe that optimizing the term
%\begin{equation*}
%	L_2 = R_0^p m^{-ps} + \frac{\delta^p m^{pt}}{N^{\frac p2}}
%\end{equation*}
%leads to inferior convergence rate w.r.t. $N$. 

Second, when $s,t>0$ satisfy $2s-t+1>0$, we have that $1/2-t\nu_3>0$ and therefore $\sqrt N m^{-t}$ grows polynomially. In consequence, the last terms on the right-hand side of inequality \eqref{eq:Lp_upper_bound} decay exponentially w.r.t. $N$, and we find that $a_{N, R_0, \delta}^p$ dominates the expectation up to a constant with parameter choice rule indicated by \eqref{eq:params_nu} asymptotically w.r.t $N$.
This yields the result.
\hfill\qed

\section{Minimax optimality}

\subsection{Preliminaries}

In this section we follow the main steps of strategy devised in \cite{blanchard2018optimal} for proving minimax optimality. We note that here the proof is in some parts simplified since our source condition is not dependent on the normal operator $B_\dmeas$ and, therefore, allowing more explicit arguments directly deriving the strong minimax optimality.

Let us construct the necessary set of tools for the proof. For the moment, consider a general model ${\mathcal P}' = \{P_\theta \; | \; \theta \in \Theta\}$ of probability measures on a measurable space $(Z, {\mathcal A})$. Further, let $d: \Theta \times \Theta \to [0,\infty)$ be a metric.

Now suppose $P_1, P_2 \in {\mathcal P}'$ and recall the definition of Kullback--Leibler divergence between $P_1$ and $P_2$ given by
\begin{equation*}
	D_{KL}(P_1, P_2) = \int \log \left(\frac{dP_1}{dP_2}\right) dP_1,
\end{equation*}
if $P_1$ is absolutely continuous w.r.t. $P_2$.
For a $n$-fold tensors, i.e. $P_j^{\otimes N} = P_j \otimes ... \otimes P_j$ on $Z_n = Z \otimes ... \otimes Z$, we have
\begin{equation*}
	D_{KL}(P_1^{\otimes N}, P_2^{\otimes N}) = N D_{KL}(P_1, P_2).
\end{equation*}

Let us briefly describe the procedure to prove (strong) minimax optimality below. For any $N\in \N$ large enough we aim to find $\epsilon = \epsilon(N)$ with the following properties:
we can find $K_\epsilon = K_{\epsilon(N)}$ parameters $\theta_1, ..., \theta_{K_\epsilon} \in \Theta$ such that $\theta_i$ and $\theta_j$, $i\neq j$ are $\epsilon$-separated with respect to the associated distance while the (data-generating) distributions $P_{\theta_i} = P_i \in {\mathcal P}'$ have small mutual Kullback--Leibler divergence. 
We utilize the following fundamental lower bound
\begin{equation}
	\label{eq:aux_minimax1}
	\inf_{\hat \theta} \sup_{P_\theta \in {\mathcal P}'} \left(\E_\theta d(\theta, \hat \theta)^p\right)^{\frac 1p}
	\geq \epsilon \inf_{\hat \theta} \sup_{P_\theta \in {\mathcal P}'} 
	\Prob \left(d(\theta, \hat\theta) \geq \epsilon\right)
	\geq \epsilon \inf_{\hat \theta} \sup_{1\leq j \leq K_\epsilon} \Prob_j\left(d(\theta, \hat\theta) \geq \epsilon \right).
\end{equation}
To prove that $a_{N, R_0, \delta}$ satisfies the claim in theorem \ref{thm:minimax_optimality}, we show that $\epsilon = \epsilon(N)$ can be chosen so that
\begin{equation}
	\label{eq:aux_minimax2}
	\frac{\epsilon(N)}{a_{N, R_0, \delta}}>C \quad \text{and} \quad \inf_{\hat \theta} \sup_{1\leq j \leq K_\epsilon}  \Prob_j\left(d(\hat \theta, \theta) \geq \epsilon \right) > C
\end{equation}
hold simultaneously for some positive constant independent of $N$. 
This result will provide the strong minimax lower bound described in theorem \ref{thm:minimax_optimality} while the upper bound is obtained in theorem \ref{thm:main_result_exp}.

To begin with, let us record the main auxiliary results that will be utilized.

\begin{proposition}[{\cite[Prop. 6.1]{blanchard2018optimal}}]
\label{prop:minimax_blan}
Assume that $K\geq 2$ and suppose $\Theta$ contains $K+1$ elements $\theta_0, ..., \theta_K$ such that
\begin{itemize}
\item[(i)] For some $\epsilon>0$ and for any $0\leq i < j\leq K$ we have $d(\theta_i, \theta_j) \geq 2\epsilon$,
\item[(ii)] For any $j=1, ..., K$, $P_j$ is absolutely continuous with respect to $P_0$ and
\begin{equation*}
	\frac 1K \sum_{j=1}^K D_{KL} (P_j, P_0) \leq \omega \log K
\end{equation*}
for some $0<\omega<1/8$. 
\end{itemize}
Then it follows that
\begin{equation*}
	\epsilon \inf_{\hat \theta} \sup_{1\leq j \leq K} P_j(d(\hat \theta, \theta_j) \geq \epsilon) \geq
	\frac{\sqrt{K}}{1+\sqrt K} \left(1-2\omega - \sqrt{\frac{2\omega}{\log K}}\right) > 0
\end{equation*}
\end{proposition}

In addition, the following lemma is key in constructing the required elements $\theta_1, ..., \theta_{K_\epsilon}$.

\begin{lemma}[{\cite[Prop. 6]{caponnetto2007optimal}}]
\label{lem:opt_pis}
For any $k\geq 28$ there exists an integer $N_k>3$ and $\pi_1, ..., \pi_{N_k} \in \{-1, +1\}^k$ such that
for any $i,j\in \{1, ..., N_k\}$ with $i\neq j$ it holds
\begin{equation*}
	\log(N_k - 1) > \frac k{36} > \frac 23
\end{equation*}
and 
\begin{equation*}
	\sum_{\ell=1}^k (\pi_i^\ell - \pi_j^\ell)^2 \geq k,
\end{equation*}
where $\pi_i = (\pi_i^1, ..., \pi_i^k)$.
\end{lemma}

\subsection{Proof of strong minimax optimality}

Let us now turn to the concrete problem at hand. 
We consider as $d : {\mathcal H} \times {\mathcal H}\to \R$ the metric induced by norm, i.e., $d(f_1, f_2) = \norm{f_1 - f_2}_{{\mathcal H}}$. For $\nu \in {\mathcal P}^<(\cthree,t)$ we associate the following joint measure
\begin{equation}
	\label{eq:joint_measure}
	\jmeas_f(dx, dy) = \likemeas_f(dy | x) \nu(dx) \quad\text{on}\; \dom\times \R,
\end{equation}
where $\likemeas_f(dy | x) = {\mathcal N}(S_x f, \delta^2)$. We observe that if $f \in \Theta(s,R)$, then $\jmeas_f \in {\mathcal M}(s,R,{\mathcal P}^<(\cthree,t))$. Moreover, by \cite[Prop. 6.2]{blanchard2018optimal} for $\jmeas_{f_1}, \jmeas_{f_2} \in {\mathcal M}(s,R,{\mathcal P}^<(\cthree,t))$ we have
\begin{equation*}
	D_{KL}(\jmeas_{f_1}, \jmeas_{f_2}) = \frac{1}{2\delta^2} \norm{B_\dmeas^{\frac 12}(f_1 - f_2)}_{{\mathcal H}}^2.
\end{equation*}

We assume that $\{e_j\}_{j=1}^\infty \subset {\mathcal H}$ form an orthonormal basis with the property that $e_m \in V_m$ for any $m\in \N$ and $\{e_j\}_{j=1}^m$ forms an orthonormal basis for the subspace $V_m$. 
Such basis can be constructed e.g. by the Gram--Schmidt orthonormalization procedure.

%Let us define the set
%\begin{equation*}
%	\Omega(r,R) := \{f \in {\mathcal H} \; | \; f = \sum_{\ell=1}^\infty \ell^{-r} \langle g, e_\ell\rangle_{{\mathcal H}}e_\ell, \quad \norm{g}_{{\mathcal H}} < R\} \subset {\mathcal H}.
%\end{equation*}
%The immediate consequence is the following lemma.
%\begin{lemma}
%We have that $\Omega(s,R) \subset \Theta(s, R)$.
%\end{lemma}
%\begin{proof}
%The claim follows by observing that
%\begin{equation*}
%	\norm{(I-P_m)f}_{{\mathcal H}}^2 = \sum_{\ell>m} \ell^{-2\frac st} \langle g, e_j \rangle_{{\mathcal H}}^2 \leq (m+1)^{-2s} R^2
%\end{equation*}
%for all $m\in \N$.
%\end{proof}
%

\begin{proposition}
\label{prop:minimax_main_prop}
Assume that $\dmeas \in {\mathcal P}^<(t,\cthree)$ and let $s,R>0$. For any $0< \epsilon \leq \epsilon_0$ with $\epsilon_0 = 56^{-s} R$ there exists $K_\epsilon \in \N$ and functions $f_1, ..., f_{K_\epsilon} \in {\mathcal H}$ satisfying following three conditions:
\begin{itemize}
	\item[(i)] It holds that $f_i \in \Theta(s,R)$ and
	\begin{equation*}
		\norm{f_i -f_j}_{{\mathcal H}} > \epsilon
	\end{equation*}
	for any $i,j=1, ..., K_\epsilon$ with $i\neq j$.
	\item[(ii)] Let $\jmeas_i = \jmeas_{f_i}$ be given by \eqref{eq:joint_measure}. Then it holds
	\begin{equation*}
		D_{KL}(\jmeas_i, \jmeas_j) \leq C R^2 \delta^{-2} \left(\frac \epsilon R\right)^{2 + \frac ts}
	\end{equation*}
	for any $i,j = 1, ..., K_\epsilon$ with $i\neq j$, where the constant $C$ is dependent on $t$.
	\item[(iii)] It holds that $\log(K_\epsilon - 1) \geq C \left(\frac R\epsilon\right)^{\frac 1{s}}$.
\end{itemize}
\end{proposition}
\begin{proof}
We define
\begin{equation}
	k = \frac 12 \left(\frac R \epsilon\right)^{\frac 1{s}} 
\end{equation}
and observe that with $\epsilon\leq \epsilon_0$ we have $k\geq 28$ satisfying condition of lemma \ref{lem:opt_pis}. In consequence, let $K_\epsilon = N_{k(\epsilon)}>3$ and $\pi_1, ..., \pi_{K_\epsilon} \in \{-1, +1\}^k$ be given by the same lemma
and define
\begin{equation*}
	f_i = \frac{\epsilon}{\sqrt k} \sum_{\ell=k+1}^{2k} \pi_i^{(\ell-k)} e_\ell
\end{equation*}
for $i=1, ..., K_\epsilon$.
We observe that $f_i \in \Omega(s,R)$ since one can write
$f_i =  \sum_{\ell=k+1}^{\infty} \ell^{-s} \langle f_i,  \ell^s e_\ell\rangle_{{\mathcal H}} e_\ell$
and show that
\begin{eqnarray*}
	\norm{(I-P_m)f_i}_{{\mathcal H}}^2 & = & \sum_{\ell=\max\{k+1,m+1\}}^{\infty} \ell^{-2s} \langle f_i,  \ell^s e_\ell\rangle_{{\mathcal H}}^2 \\
	& \leq & (m+1)^{-2s} \sum_{\ell=\max\{k+1,m+1\}}^{\infty} \langle f_i,  \ell^s e_\ell\rangle_{{\mathcal H}}^2 \\
	& \leq & (m+1)^{-2s}\frac{\epsilon^2}k \sum_{\ell= k+1}^{2k} \ell^{2s} \leq (2k)^{2s}\epsilon^2 (m+1)^{-2s} = R^2 (m+1)^{-2s}.
\end{eqnarray*}
Moreover, we have
\begin{equation*}
	\norm{f_i - f_j}_{{\mathcal H}}^2 = \sum_{\ell=k+1}^{2k} \ell^{-2s} |\langle g_i-g_j, e_\ell\rangle_{{\mathcal H}}|^2 = \frac{\epsilon^2}k \sum_{\ell=k+1}^{2k} |\pi_i^{(\ell-k)} - \pi_j^{(\ell-k)}|^2\geq \epsilon^2,
\end{equation*}
which completes the proof for claim (i).

Consider now claim (ii). We have that
\begin{eqnarray*}
	D_{KL}(\jmeas_i, \jmeas_j) & = & \frac 1{2\delta^2} \norm{B_\dmeas^{\frac 12} (f_i - f_j)}_{{\mathcal H}}^2 \\
	& = & \frac{\epsilon^2}{2\delta^2 k} \sum_{\ell=k+1}^{2k} |\pi_i^{(\ell-k)} - \pi_j^{(\ell-k)}|^2 \langle e_\ell, B_\dmeas e_\ell\rangle_{{\mathcal H}} \\
	& \leq & \frac {C\epsilon^2}{2\delta^2 k}  \sum_{\ell=k+1}^{2k} \ell^{-t} |\pi_i^{(\ell-k)} - \pi_j^{(\ell-k)}|^2 \\
	& \leq & \frac{2 C \epsilon^2}{\delta^2} \cdot k^{-t} \\
	& = & \frac{2 C \epsilon^2}{\delta^2} \cdot 2^{t} \left(\frac{\epsilon}R\right)^{\frac{t}{s}}
	= 2^{t+1}C \delta^{-2} R^2 \left(\frac{\epsilon}R\right)^{\frac{t}{s}+2}
\end{eqnarray*}

For the claim (iii), it remains to observe that
\begin{equation}
	\log(K_\epsilon - 1) \geq \frac k{36} = \frac 1{72} \left(\frac R \epsilon\right)^{\frac 1s},
\end{equation}
which concludes the proof.
\end{proof}

Now we are ready to prove the minimax optimality.\\

{\bf Proof of theorem \ref{thm:minimax_optimality}.} Let us now fix parameters $s,t,R,\cone, \ctwo,\cthree>0$
and consider model ${\mathcal M}(\Theta', {\mathcal P}')$, where
\begin{equation*}
	 {\mathcal P}' = \Big\{\dmeas \in {\mathcal P} \;  \Big| \; \dmeas \in {\mathcal P}^{>}(t,\cone) \cap {\mathcal P}^\times(\ctwo) \cap {\mathcal P}^{<}(t, \cthree)\Big\} 
\end{equation*}
and $\Theta' = \Theta(s,R_0)$ given by \eqref{eq:source} parametrized by admissible subspaces $V_m \subset {\mathcal H}$, $m\geq 1$. By theorem \ref{thm:main_result_exp} we know that 
\begin{equation*}
	a_{N, R_0, \delta} = R_0 \left(\frac{\delta}{R_0\sqrt N}\right)^{\frac{2s}{2s+t+1}}.
\end{equation*}
yields an upper rate of convergence in $L^p$. It remains to show that $a_{N, R_0, \delta}$ is also a strong minimax lower rate of convergence. 

To this end, we construct a rule $N = N(\epsilon)$ (which we will invert below to obtain $\epsilon = \epsilon(N)$) for $\epsilon$ small enough that yields the lower bound in the spirit of \eqref{eq:aux_minimax1} and \eqref{eq:aux_minimax2}. For $\epsilon \leq 56^{-s}R$ let $K_\epsilon$ and $f_1, ..., f_{K_\epsilon} \in \Omega(s,R) \subset \Theta(s,R)$ be given by proposition \ref{prop:minimax_main_prop}. Let us then consider the conditions of proposition \ref{prop:minimax_blan}. Clearly, the condition (i) is satisfied due to first result in proposition \ref{prop:minimax_main_prop}. For the second condition we have
\begin{eqnarray*}
	\frac{1}{K_\epsilon-1} \sum_{j=1}^{K_\epsilon-1} D_{KL}(\jmeas_j^{\otimes N}, \jmeas_{K_\epsilon}^{\otimes N})
	& = & \frac{N}{K_\epsilon -1}\sum_{j=1}^{K_\epsilon-1} D_{KL}(\jmeas_j, \jmeas_{K_\epsilon}) \\
	& \leq & N C_t R^2 \delta^{-2} \left(\frac{\epsilon}{R}\right)^{2+\frac ts} \\
	& \leq & N C_t R^2 \delta^{-2} \left(\frac{\epsilon}{R}\right)^{2+\frac{t+1}s}  \log(K_\epsilon-1)\\
	& =: & \omega \log(K_\epsilon-1).
\end{eqnarray*}
By setting
\begin{equation}
	\label{eq:minimax_n(eps)}
	N(\epsilon) = \left\lfloor \left(8C_t R^2 \delta^{-2} \left(\frac{\epsilon}{R}\right)^{2+\frac{t+1}s}\right)^{-1} \right\rfloor
\end{equation}
ensures $\omega \leq 1/8$ and condition (ii).

Now we obtain by proposition \ref{prop:minimax_blan} that
\begin{eqnarray*}
	\inf_{\hat f} \max_{1\leq j \leq K_\epsilon} \jmeas_j^{\otimes N}	\left(\norm{\hat f - f_j}_{{\mathcal H}} \geq \frac \epsilon 2\right) \geq \frac{\sqrt{K_\epsilon-1}}{1 + \sqrt{K_\epsilon-1}} \left(1 - 2\omega - \sqrt{\frac{2\omega}{\log(K_\epsilon-1)}} \right) \geq C > 0,
\end{eqnarray*}
where the constant $C$ is independent of $\epsilon$.

Finally, inverting identity \eqref{eq:minimax_n(eps)} translates into a bound
\begin{equation*}
	\epsilon = \epsilon(N) \geq \tilde C_t R \left(\frac{\delta\sqrt N}{R}\right)^{\frac{2s}{2s+t+1}},
\end{equation*}
which is aligned with the rate $a_{N,R_0,\delta}$ implying $\epsilon(N)/a_{N,R_0,\delta}$ is bounded away from zero and, in conclusion, we have
\begin{equation*}
	\liminf_{n\to\infty} \inf_{\hat f} \sup_{\jmeas \in {\mathcal M}'} \frac{ \left(\E_{\bdmeas_N} \norm{\hat f - f^\dagger}_{{\mathcal H}}^p\right)^{\frac 1p}}{a_{N, R_0, \delta}} > 0.
\end{equation*}
This completes the proof.
\hfill\qed

\section{Conclusions}

In this work we have studied statistical inverse learning for a regularization strategy obtained by projecting the unknown to finite-dimensional subspaces. We have demonstrated that our nonlinear estimator, which is constructed as a norm cut-off of a linear maximum likelihood estimator, achieves minimax optimal convergence rates. Indeed, in the particular example of truncated singular value decomposition, our rate coincides with the known minimax optimal rate.

Projection methods are often motivated by iterative schemes, where the number of iteration steps identifies the dimension of a subspace to which the unknown is projected. Here, we require that the subspace structure $\{V_m\}$ is fixed. It remains future work to study interesting iterative methods for which conditions such as imposed by the set ${\mathcal P}^{</>}$ are satisfied with high probability; e.g. if the structure $\{V_m\}$ is dependent on the observational data $\by$, is data-driven or a spectral basis is approximated by a power method. Moreover, future work naturally includes extending our results to nonlinear inverse problems. 

%{\bf Acknowledgements.} The author was supported by the Academy of Finland through decision number 326961. 

\bibliographystyle{plain}
\bibliography{references}

\end{document}